\documentclass{amsart}
\usepackage{amssymb}
\usepackage{fullpage}
\usepackage{amsthm}
\usepackage{bm}
\usepackage{latexsym}
\usepackage{float}
\restylefloat{table}
\usepackage{amsmath}
\usepackage{multirow}
\usepackage{eufrak}
\usepackage{mathrsfs}
\usepackage[font={small,it}]{caption}
\usepackage{amscd}
\usepackage[all,cmtip]{xy}
\usepackage[usenames]{color}
\usepackage{graphicx}
\usepackage{color}
\usepackage{amscd}
\usepackage{float}
\usepackage{graphics}
\usepackage{tikz}
\usepackage{tikz-cd}
\usepackage{comment}
\usepackage{xspace}
\usepackage{mathtools}
\usepackage{lipsum}
\usepackage{tabularx}
\usepackage{booktabs}
\usepackage{amssymb}
\usepackage{amsthm}
\usepackage{graphicx}
\usepackage{subfig}
\usepackage{enumerate}

\usepackage[style=alphabetic, backend=bibtex,
 doi=false,isbn=false,url=false,
 eprint=true, sorting=nyt, maxnames=99, backref=true]{biblatex}
\usepackage{hyperref}
\usepackage[capitalize]{cleveref}

\usetikzlibrary{patterns,decorations.pathreplacing}
\usepackage{marginnote}

\theoremstyle{plain}

\newtheorem{theorem}{Theorem}[section]
\newtheorem{proposition}[theorem]{Proposition}
\newtheorem{lemma}[theorem]{Lemma}

\newtheorem{question}[theorem]{Question}

\theoremstyle{definition}

\newtheorem{definition}[theorem]{Definition}

\newtheorem{observation}[theorem]{Observation}

\theoremstyle{remark}

\newtheorem{remark}[theorem]{Remark}
\newtheorem{example}[theorem]{Example}

\DeclareFontFamily{U}{wncy}{}
    \DeclareFontShape{U}{wncy}{m}{n}{<->wncyr10}{}
    \DeclareSymbolFont{mcy}{U}{wncy}{m}{n}
    \DeclareMathSymbol{\Sha}{\mathord}{mcy}{"58}

\DeclareMathOperator{\SL}{SL}
\DeclareMathOperator{\Ind}{Ind}

\def\R{{\mathbb{R}}}
\def\Z{{\mathbb{Z}}}
\def\F{{\mathbb{F}}}
\def\Q{{\mathbb{Q}}}
\def\Qbar{{\overline{\Q}}}
\def\C{{\mathbb{C}}}
\def\P{{\mathbb{P}}}
\def\Pdual{{\check{\mathbb{P}}}}

\DeclareMathOperator{\ord}{ord}
\DeclareMathOperator{\divi}{div}

\def\calA{{\mathcal{A}}}

\def\V{{\mathcal{V}}}
\usepackage{microtype}

\DeclareMathOperator{\Vis}{Vis}
\DeclareMathOperator{\Br}{Br}

\DeclareMathOperator{\Gal}{Gal}

\DeclareMathOperator{\Pic}{Pic}
\DeclareMathOperator{\Jac}{Jac}
\DeclareMathOperator{\Res}{Res}

\DeclareMathOperator{\Cl}{Cl}
\newcommand{\Mod}[1]{\ (\mathrm{mod}\ #1)}
\newcommand{\disc}{\mathrm{disc}}

\newcommand{\shiva}[1]{{\textcolor{red}{Shiva: #1}}}

\usepackage{xcolor}

\newenvironment{barinder}
{\color{orange}\par\noindent\textbf{Barinder: }}
{\par}

\begin{document}
\title{On the Visibility category of the Shafarevich--Tate group}

\author[B.S. Banwait]{Barinder S. Banwait}
\address{Barinder S. Banwait\\
London\\
United Kingdom
}
\email{barinder.s.banwait@gmail.com}

\author{Jerson Caro}
\address{Jerson Caro\\
Departamento de Matemáticas\\
Universidad de los Andes\\
Bogotá\\
Colombia
}
\email{jl.caro10@uniandes.edu.co}

\author{Shiva Chidambaram}
\address{Shiva Chidambaram\\
Department of Mathematics \& Statistics\\  
Madison Wisconsin\\
USA}
\email{chidambaram3@wisc.edu}


\begin{abstract}
    Given an elliptic curve $E$ over $\Q$ and a nontrivial element $\sigma$ of its Shafarevich--Tate group $\Sha(E)$, we introduce the \textbf{Visualization category} $\V(E; \sigma)$ of abelian varieties that ``visualize'' $\sigma$ in the sense of Mazur, and we study minimal objects in this category. In particular, we show that there can be several minimal visualizing abelian varieties of different dimensions, answering a question of Mazur. We revisit two constructions of visualizing abelian varieties: restriction of scalars (as in the work of Agashe and Stein), and a construction due to de Jong (as in the work of Cremona and Mazur). We show that restriction of scalars typically produces minimal visualizations. When $\sigma$ has order $2$ or $3$, we build upon the de Jong construction and make it totally explicit. While the de Jong construction can produce non-minimal objects, an appropriate choice in the construction for order $2$ elements $\sigma$ yields an explicit genus $2$ curve whose Jacobian is a minimal visualization. For order $3$ elements we apply our algorithmic construction to Fisher's database of such elements, and obtain computational evidence that, in the absence of a $3$-isogeny, the de Jong construction yields a minimal visualization.
\end{abstract}

\maketitle

\section{Introduction}

\subsection{Background to Visibility}

Associated to an elliptic curve $E$ over a number field $K$ are two groups of particular arithmetic interest: the \textsf{Mordell--Weil group} $E(K)$, and the \textsf{Shafarevich--Tate group} $\Sha(E/K)$ (henceforth denoted $\Sha(E)$ if the base field is understood). Although these two groups are intimately connected, it is easier to exhibit elements of $E(K)$ than it is to exhibit elements of $\Sha(E)$. This was the motivation for Mazur \cite{Mazur} to introduce the notion of \emph{visualizing} the elements of $\Sha(E)$. We briefly recall the definitions.

A \textsf{torsor} (or \textsf{principal homogeneous space}) under $E$ is a pair $(C, \mu)$ where $C$ is a genus-$1$ curve, and $\mu : E \times C \to C$ is a morphism inducing a simply transitive action on $\overline{K}$-points. We may view $C$ as a twist of $E$ by a cocycle taking values in $E$ (acting on itself by translations) and thus $C$ may be regarded as an element of the \textsf{Weil--Ch\^{a}telet group} $H^1(K,E)$. The Shafarevich--Tate group $\Sha(E)$ is then the subgroup of $H^1(K,E)$ consisting of torsors that are everywhere locally soluble; that is, the curve $C$ has a $K_v$-point for all places $v$.

Mazur sought to more concretely ``visualize'' the curve $C$ as a subcurve of an abelian variety. Specifically, if we have an inclusion $\iota: E \to A$ of $E$ into an abelian variety $A$, and we regard $C$ as the twist of $E$ by the element $\sigma \in H^1(K,E)$, then we may consider the commutative diagram
\[
\begin{tikzcd}
E \arrow[r] \arrow[d, "\sigma" left, "\cong" right] & A \arrow[d, "\cong" left, "\iota_\ast(\sigma)" right] \\
C \arrow[r] & V.
\end{tikzcd}
\]
Here, $\iota_\ast : H^1(K,E) \to H^1(K,A)$ is the natural push-forward map, and $V$ is the twist of $A$ by $\iota_\ast(\sigma)$. This yields an inclusion $C \to V$, and one might say that $C$ is \textsf{visible} in $V$. In the situation where in fact $V$ is merely the trivial twist of $A$ (hence isomorphic to $A$ itself), one thus says that $C$ is visible in $A$, and we define the subgroup of $H^1(K,E)$ \textsf{visible in $A$} as \[ \Vis_AH^1(K,E) := \ker\left(H^1(K, E) \overset{\iota_\ast}\longrightarrow H^1(K, A)\right). \]

By considering Galois cohomology of the exact sequence \[ 0\to E \to A \to B \to 0,\] where $B$ is the quotient abelian variety $A/E$, two equivalent formulations of visibility are readily obtained, namely that $C$ is visible in $A$ if and only if
\begin{enumerate}
    \item $C$ is a translate of $E$ inside $A$ by a point in $A(\overline{K})$ that maps to a $K$-rational point on $B$;
    \item $C$ is the fiber above a $K$-rational point on $B$ under the projection $A \to B$, i.e., the Mordell-Weil group $B(K)$ \emph{explains} $C$.
\end{enumerate}

We define $\Vis_A\Sha(E) := \Vis_AH^1(K,E) \cap \Sha(E)$ to be the elements of $H^1(K,E)$ visible in $A$ that are also everywhere locally soluble.

Cremona and Mazur \cite{CremonaMazur} proved that every element $\sigma$ of $H^1(K,E)$ is visible in some abelian variety, via a construction (to be found in Remark 1/Theorem 2 of \emph{loc. cit.}) that they attribute to Johan de Jong in Remark 3 of \emph{loc. cit.}; we thus refer to this as the \emph{de Jong construction} throughout this paper. We note that this construction does not give an explicit visualizing abelian variety.

This gives rise to the question of identifying the \emph{smallest} such visualizing abelian variety. Most of the work in this area up to now has taken the interpretation of ``smallest'' to mean, for a given $\sigma \in \Sha(E)$, (a) finding the minimum of the dimensions of the abelian varieties that visualize $\sigma$, known in the literature as the \textsf{visibility dimension} of $\sigma$; and (b) explicitly exhibiting a visualizing abelian variety of this visibility dimension. \Cref{tab:literature_review} gives a brief literature review of the works in this area up to now.

\begin{table}[htbp]
\centering
\caption{Literature review on Visibility of $\Sha$ up to 2018}
\label{tab:literature_review}
\begin{tabularx}{\textwidth}{l X} 
\toprule
\textbf{Authors (year)} & \textbf{Summary of work} \\
\midrule
Mazur (1999) \cite{Mazur} & Elements in $\Sha(E)[3]$ have visibility dimension 2 \\
\addlinespace 
Cremona--Mazur (2000) \cite{CremonaMazur} & Proved every element $\sigma \in H^1(K,E)$ is visible in some abelian variety (via a method attributed to de Jong); and computational investigation into visualizing elements of $\Sha$ in abelian surfaces\\
\addlinespace
Klenke (2001) \cite{Klenke} & Elements in $H^1(K,E)[2]$ have visibility dimension 2 \\
\addlinespace
Agashe--Stein (2002) \cite{AgasheStein} & Explicit bound on visibility dimension of any $\sigma \in H^1(K,E)$ \\
\addlinespace
Bruin (2004) \cite{Bruin04} & Explicit construction of visualizing genus $2$ Jacobians for elements in $\Sha(E)[2]$ \\
\addlinespace
Bruin--Dahmen (2010) \cite{BruinDahmen} & Explicit construction of visualizing genus $2$ Jacobians for elements in $\Sha(E)[3]$ \\
\addlinespace
Fisher (2014) \cite{Fisher} & First examples of \emph{invisibility}: exhibited elements in $\Sha$ of order 6 and 7 that cannot be visualized in an abelian surface \\
\addlinespace
Fisher (2016) \cite{Fisher_vis7} & Explicit constructions of visualization in abelian threefolds\\
\addlinespace
Bruin--Fisher (2018) \cite{BruinFisher} & Examples of elements of $\Sha(E)[4]$ invisible in \emph{principally polarised} abelian surfaces\\
\bottomrule
\end{tabularx}
\end{table}

\subsection{The Visibility Category}

In this paper, however, we take a more categorical approach to the question of what the ``smallest'' visualizing abelian variety could be, one that was suggested to us by Mazur. To describe this, we make the following definition.

\begin{definition}
    Let $A$ be an abelian variety over a number field $K$, and let $\sigma \in \Sha(A)$. 
    We define the \textsf{$\sigma$-visibility category} $\mathcal{V}(A/K, \sigma)$ 
    to be the category with objects and morphisms as follows:

    \begin{list}{}{\leftmargin=3.5em \itemindent=0em \labelsep=1em}
        \item[\textbf{Objects:}] Pairs $(B/K, \iota)$ consisting of
        \begin{itemize}
            \item an abelian variety $B$ over $K$,
            \item an inclusion $\iota : A \hookrightarrow B$,
        \end{itemize}
        such that $\sigma$ is visible in $B$.

        \medskip

        \item[\textbf{Morphisms:}] Injective homomorphisms of abelian varieties 
        $B \to B'$ that are compatible with the inclusions 
        $\iota$ and $\iota'$, that is, commutative diagrams \[
    \begin{tikzcd}
    B \arrow[hookrightarrow, rr] && B' \\
    & A. \arrow[hookrightarrow, ul] \arrow[hookrightarrow, ur]
    \end{tikzcd}
        \]
    \end{list}

We similarly define, for $n \geq 2$ an integer, $\mathcal{V}(A/K, n)$ to be the \textsf{$n$-visibility category}, where we further ask that the objects visualize all elements in $\Sha(A)[n]$ (and not merely a given one $\sigma$).
\end{definition}

One may then consider the smallest visualizing abelian varieties to be the minimal objects in these categories; that is, we say that $(B/K, \iota)$ is \textbf{minimal} if, for every morphism \[
    \begin{tikzcd}
    B_0 \arrow[hookrightarrow, rr] && B \\
    & A, \arrow[hookrightarrow, ul] \arrow[hookrightarrow, ur]
    \end{tikzcd}
        \]
the homomorphism $B_0 \to B$ is an isomorphism of abelian varieties.

It is easily established (see \Cref{sec:prelims}) that both of these categories are nonempty, and furthermore that in both of these categories, if the quotient abelian variety $B/A$ is simple over $K$, then $B$ is a minimal visualization in the respective category.

\subsection{Some fundamental questions about minimal objects in the Visibility category}

We consider the following natural questions about minimal objects in the visibility category, posed to us by Mazur.

\begin{question}\label{q:different_minima}
    Can we find different minimal objects in each of $\mathcal{V}(A/K, \sigma)$ and $\mathcal{V}(A/K, n)$? If so, how different can they be, and can the dimension of these minimal objects be different?
\end{question}

\begin{question}\label{q:explicit_minimal_constructions}
    Of the various constructions of visualizing abelian varieties in the literature (see \Cref{tab:literature_review}), do any readily yield minimal objects? 
\end{question}

The next question is prompted by Mazur's observation from \cite{Mazur} that, unusually often, if an abelian variety visualizes one element of a given order $n$, then it visualizes all of $\Sha(E)[n]$\footnote{This is written in the Introduction of \cite{Mazur} as ``Quite often, for a given elliptic curve $E$, it was even the case that all the elements of its Shafarevich--Tate group were representable by curves of genus $1$ in a \emph{single abelian surface} ...'' (emphasis added).}.

\begin{question}\label{q:difference_between_categories}
    Can it happen that there is an object in $\mathcal{V}(A/K, \sigma)$ that is not an object in $\mathcal{V}(A/K, n)$?
\end{question}

\subsection{Answers to the fundamental questions about minimal objects in the Visibility category}

In this paper we address the aforementioned questions in the following way. We answer \Cref{q:different_minima} affirmatively by exhibiting an elliptic curve that admits two different minimal visualizations of $\Sha(E)[5]$ of dimensions 2 and 3.

\begin{theorem}\label{thm:different_minima}
    Let $E$ be the elliptic curve with LMFDB label \href{https://www.lmfdb.org/EllipticCurve/Q/161472/bz/1}{\texttt{161472.bz.1}}: \[ y^2=x^3-x^2-169321x-28379327.\]

    \begin{enumerate}
        \item\label{item:2d_miminal_vis} There is an abelian surface $A$ over $\Q$ that visualizes $\Sha(E)[5]$, and we have \[ A \cong (E \times F) / \Delta, \] where $F$ is the elliptic curve \[ y^2 = x^3 - x^2 - 2285384981x - 42050896792563, \] and $E[5] \cong \Delta \cong F[5]$, diagonally embedded in each component of $E \times F$.
        \item\label{item:3d_miminal_vis} There is an abelian threefold $B$ over $\Q$ that visualizes $\Sha(E)[5]$, and we have \[ B \cong (E \times \Jac(C))/\Delta, \] where $C$ is the genus~$2$ curve given by \[ -2y^2 = x^6 + 2x^4 + 12x^3 + 5x^2 + 6x + 1,\] and $E[5]$ is isomorphic to a subgroup $\Delta'$ of $\Jac(C)[5]$, and where we regard $\Delta$ as diagonally embedded in each component of $B$.
        \item Both $A$ and $B$ are minimal in $\mathcal{V}(E/\Q, 5)$.
    \end{enumerate}    
\end{theorem}

Finding explicit examples of the sort given in \Cref{thm:different_minima} required several inputs:

\begin{itemize}
    \item Obtaining from the LMFDB \cite{lmfdb} a list of elliptic curves $E$ with $\Sha(E)[5] \ne \{0\}$, satisfying certain other favourable conditions coming from a result of Fisher \cite{Fisher_vis7}.
    
    \item For each curve $E$ found in the previous step, generating genus $2$ curves $C$ over $\Q$ whose Jacobian has real multiplication by $\sqrt{5}$ and $\Jac(C)[\sqrt{5}] \cong E[5]$; the main input here is the explicit models for genus $2$ Jacobians with RM by $\sqrt{5}$ \cite{cowan2024generic}. This provides the abelian variety $B$ as in part~\ref{item:3d_miminal_vis}, though it isn't guaranteed that $E$ will satisfy part~\ref{item:2d_miminal_vis}.

    \item From the candidate pairs $(E,C)$ from the output of the last step, we look for a $5$-congruent elliptic curve $F$ by using a result of Rubin--Silverberg \cite{rubin1995families}, such that the abelian surface $(E \times F)/\Delta$ (with $\Delta \simeq E[5] \simeq F[5]$ diagonally embedded in $E \times F$) satisfies the conditions of a result of Fisher \cite[Theorem 2.2]{Fisher_vis7} (which is a reformulation of a result of Agashe and Stein \cite[Theorem 3.1]{AgasheStein}).
\end{itemize}

Regarding \Cref{q:explicit_minimal_constructions}, we study the restriction of scalars method (that was first used in the context of visibility by Agashe and Stein \cite{AgasheStein}) more closely, and show that it provides explicit minimal objects in $\V(E/K, \sigma)$ for a given $\sigma \in \Sha(E)[n]$, for any $n$.


\begin{theorem}\label{thm:restrictionofscalars_minimal_generaln}
    Let $\sigma \in \Sha(E/K)$ be a nontrivial element of order $n$. 
    Then there exists a degree-$n$ extension $L/K$ such that the Weil restriction
    $A \coloneqq \Res_{L/K}(E_L)$
    is a minimal object in the visibility category $\mathcal{V}(E/K,\sigma)$.
\end{theorem}


On the other hand, Cremona and Mazur establish the existence of a curve $X$ -- which we call a \emph{de Jong curve} -- whose Jacobian $\Jac(X)$ visualizes $\sigma$. As they admit in their paper, their method ``does not allow us easy viewing of the curves of genus 1 that are generated''. Moreover, the curve $X$ implicitly depends on the choice of an element in a certain Azumaya algebra $\calA_\sigma$, and different choices can drastically affect the geometry of $X$. We explicitly construct a de Jong curve of genus $4$, which is not minimal in $\V(E/K, \sigma)$. 

\begin{theorem}\label{thm:cremona_mazur_curves_can_have_large_genus}
Let $E$ be the elliptic curve \href{https://www.lmfdb.org/EllipticCurve/Q/571/b/1}{571.b1}.
It is known that $\Sha(E) \cong (\Z/2\Z)^2$, and the binary quartics associated to the elements of $\Sha(E)$ as computed in \cite[Section 15.1]{Fisher_hess12} are:
\begin{align}
\begin{split}
\label{eq:binary_quartics_example}
\psi_1 &= -4x^4 - 60x^3z - 232x^2z^2 - 52xz^3 - 3z^4\\
\psi_2 &= -11x^4 - 68x^3z - 52x^2z^2 + 164xz^3 - 64z^4\\
\psi_3 &= -15x^4 - 52x^3z + 38x^2z^2 + 144xz^3 - 115z^4\\
\psi_4 &= -19x^4 + 112x^3z - 142x^2z^2 - 68xz^3 - 7z^4,
\end{split}
\end{align}
with $\psi_3$ corresponding to the trivial element. For $1 \leq i \leq 4$, let $\sigma_i \in \Sha(E)[2]$ denote the element corresponding to the binary quartic $\psi_i$.

Then $\sigma_1 \in \Sha(E)$ admits a de Jong curve $X$ of genus $4$ such that the Jacobian $\Jac(X)$ is not minimal in $\V(E, \sigma_1)$.

\end{theorem}

Despite this example, we show that it is possible to make an appropriate choice of an element in the Azumaya algebra $\calA_\sigma$ that yields, in the case that $\ord(\sigma) = 2$, a genus $2$ de Jong curve whose equation we can write down explicitly and whose Jacobian is a minimal object visualizing $\sigma$. We acheive this by using the explicit description of the Azumaya algebra $\calA_\sigma$ given by Fisher \cite{Fisher_alge19}. This establishes, in the order $2$ case, a guess of Mazur that for every $\sigma \in \Sha(E)$, the de Jong construction ought to yield minimal elements in $\V(E, \sigma)$.

\begin{theorem}\label{thm:explicit_cremona_mazur_for_n_2}
    Let $E/K$ be an elliptic curve and $0 \neq \sigma \in \Sha(E)[2]$ be represented by the binary quartic
    \begin{align}
    \label{eq:general_binary_quartic}
        ax^4 + bx^3z + cx^2z^2 + dxz^3 + ez^4, \quad \text{with} \quad 8a^2d - 4abc + b^3 \ne 0.
    \end{align}
    Let $f(x) = x^3 + cx^2 - (4ae - bd)x - 4ace + b^2 e + a d^2$ so that $y^2=f(x)$ is a model for $E$.
    Then the Jacobian of the genus $2$ curve $X$ given by the equation \[ X: y^2 = f\left(\frac{b^2 - 4a(c - x^2)}{4a} \right) \] is a $2$-dimensional minimal visualization of $\sigma$. In particular, if $E(K)/2E(K) = \{0\}$, then the complementary elliptic curve factor appearing in $\Jac(X)$ given by
    \begin{align*}
    \begin{split}
        &F : y^2 = x^3 + px^2 + qx + r, \quad \text{where}\\
        &p = -64a^3e + 16a^2bd + 16a^2c^2 - 16ab^2c + 3b^4\\
        &q = (-8ac + 3b^2)(8a^2d - 4abc + b^3)^2\\
        &r = (8a^2d - 4abc + b^3)^4
    \end{split}
    \end{align*}
    has Mordell-Weil rank at least $1$, and its Mordell-Weil group explains $\sigma \in \Sha(E)[2]$.
\end{theorem}

\begin{remark}
    \label{rmk:model_dependence}
    In \Cref{thm:explicit_cremona_mazur_for_n_2}, the genus $2$ de Jong curve $X$ and the second elliptic curve factor $F$ in $\Jac(X)$ are \emph{dependent} on the binary quartic model chosen to represent $\sigma$.
\end{remark}

\begin{remark}
    \Cref{thm:explicit_cremona_mazur_for_n_2} applies more generally to any smooth curve $C/K$ of genus $1$ given as a double cover of $\P^1$ ramified at $4$ points, as long as the curve has order $2$ in the Weil-Chatelet group $H^1(K,\Jac(C))$ and the binary quartic satisfies the condition in \Cref{eq:general_binary_quartic}. That is, the curve $C$ need not be locally soluble.
\end{remark}

\begin{remark}
    Klenke was the first to prove visibility of elements of $H^1(K,E)[2]$ in abelian surfaces. Thereafter, Bruin \cite{Bruin04} showed that elements of $\Sha(E)[2]$ can be visualized in abelian surfaces formed by gluing $E$ with suitable quadratic twists. In comparison, \Cref{thm:explicit_cremona_mazur_for_n_2} explicitly realizes visibility in abelian surfaces for a subset of $H^1(K,E)[2]$ containing $\Sha(E)[2]$. Moreover, the abelian surfaces are formed by gluing $E$ with another elliptic curve $F$, which in general is not a quadratic twist of $E$.
\end{remark}

For order $3$ elements in $\Sha(E)$, in the case where $E$ does not admit a rational $3$-isogeny, we give an explicit de Jong curve of genus $4$ in \Cref{subsec:explicit_cremona_mazur_for_n_3}. Using Fisher's database \cite{Fisher_Database_Order3Sha}, we verified for each of the $90{,}932$ ternary cubics arising from such curves of conductor $N<300{,}000$ that the Jacobian of this de Jong curve yields a minimal visualization, establishing the following theorem.

\begin{theorem}\label{thm:minimality_CremonaMazur_order3}
Let E be an elliptic curve of conductor less than $300{,}000$ that does not admit a rational $3$-isogeny, and for which $\Sha(E)[3] \neq \left\{0\right\}$. Then for every ternary cubic representing a pair of non-zero elements $\{ \pm \sigma \}$ in $\Sha(E)[3]$, the Jacobian of the de Jong curve constructed in $\Cref{subsec:explicit_cremona_mazur_for_n_3}$ is a minimal visualization of $\sigma$ and $-\sigma$.
\end{theorem}

This is evidence towards Mazur's guess that the de Jong construction ought to yield minimal visualizations for any $\sigma \in \Sha(E)$. We emphasize that this \textbf{remains open} for $|\sigma| \geq 3$.

Finally, we answer \Cref{q:difference_between_categories} in the affirmative, by providing an explicit example of an object in $\V(E, \sigma)$ that is not an object in $\V(E,n)$ where $n=\ord(\sigma)$.

\begin{theorem}\label{thm:categories_not_equal}
    We work with the example and notations introduced in \Cref{thm:cremona_mazur_curves_can_have_large_genus}.
    Let $X_2$ be the genus 2 curve given by \Cref{thm:explicit_cremona_mazur_for_n_2} whose Jacobian minimally visualizes $\sigma_2$. Then $\Jac(X_2) \in \V(E, \sigma_2)$ but $\Jac(X_2) \notin \V(E, 2)$.
\end{theorem}

A number of our constructions -- particularly the proof of \Cref{thm:different_minima} -- require searching for suitable objects in the LMFDB \cite{lmfdb} database of elliptic curves and genus 2 curves, and executing SageMath \cite{sagemath} and Magma \cite{magma} commands on the results. As such, our work is supported by a code repository that is available here:
    \begin{center}
    \url{https://github.com/shiva-chid/VisualizingShaE/}
    \end{center}
See the \texttt{README.md} there for instructions on how to use the code to verify the results in this paper. All filenames given in the paper are relative to the top level of this repository.

The paper is organised as follows. \Cref{sec:prelims} goes over some background material needed throughout, notably recalling some of Fisher's work on this topic. \Cref{thm:different_minima} is proved in \Cref{sec:different_minima}, \Cref{thm:restrictionofscalars_minimal_generaln} is proved in \Cref{sec:restiction_minimality}, \Cref{thm:explicit_cremona_mazur_for_n_2,thm:minimality_CremonaMazur_order3} are proved in \Cref{sec:explicit_cremona_mazur_for_n_2_3}, and the details of \Cref{thm:cremona_mazur_curves_can_have_large_genus,thm:categories_not_equal} are given in \Cref{sec:cremona_mazur_curves_can_have_large_genus}.

\subsection{Acknowledgements}

The authors are deeply indebted to Barry Mazur, who defined the categories we introduce in this paper, and posed the fundamental questions about them, which gave our work motivation and direction. His encouragement and patient responses to several queries during the process have been both invaluable and inspirational to us.
We thank Tom Fisher for extensive feedback on an earlier version of the manuscript, which has greatly improved the results and exposition.
The authors also thank Jennifer Balakrishnan, Alex Bartel, Nils Bruin, Jordan Ellenberg, and John Voight for helpful correspondence and comments.


\section{Preliminaries}\label{sec:prelims}

This section collects together more detailed background on the de Jong construction \cite{CremonaMazur} and the restriction of scalars method in the context of visibility \cite{AgasheStein}.

\subsection{de Jong construction} Let $K$ be a number field, let $E/K$ be an elliptic curve, and let
$\sigma \in H^1(K,E)$. 

\begin{proposition}[{\cite[Proposition 1]{CremonaMazur}}]\label{prop:visibility}
There exists an abelian variety $J$ over $K$
containing $E$ as an abelian subvariety such that $\sigma$ is visible in $J$.
\end{proposition}

We provide a brief sketch of how the variety $J$ is obtained, since it will be built upon in the sequel.

Let $n$ be the order of $\sigma$. The class $\sigma$ may be represented by an
Azumaya algebra $\calA_\sigma$ of rank $n^2$ over the function field
$F = K(E)$. There is a maximal commutative subalgebra
$L \subset \calA_\sigma$ of degree $n$ over $F$, such that the field extension $L/F$
corresponds to a degree $n$ morphism $\pi \colon X \to E$
of smooth projective curves, which is totally ramified at at least one point of $E$.

The existence of such a ramification implies that the induced map on Jacobians $\pi^\ast \colon E \to \Jac(X)$
is injective. Moreover, by construction, the Azumaya
algebra $\calA_\sigma \otimes_F L$ splits, so the class $\sigma$ becomes trivial over $L$. Equivalently, $\sigma$ is visible in $\Jac(X)$.

In \cite{CremonaMazur}, the existence of a maximal commutative subalgebra
$L \subset \calA_\sigma$ with such characteristics is proven. Following the proof, this degree $n$ extension of $K(E)$ is not unique and non-explicit. Now, we follow \cite{Fisher_alge19} in making the Azumaya algebra $\calA_\sigma$ explicit in the cases $n=2$ and $n=3$.

\noindent
\textbf{Explicit construction of Azumaya algebras for $\bf{n=2}$.} Let $0 \ne \sigma \in \Sha(E)[2]$ and let
\[
f(x,z) = ax^4 + bx^3z + cx^2z^2 + dxz^3 + ez^4 \in K[x,z]
\]
be a binary quartic associated with $\sigma$.
Following Haile and Han \cite{HaHa07}, one associates to $\sigma$ a $K$-algebra $\calA_\sigma$
defined by generators and relations encoding the formal identity
\[
(\alpha^2 r + \alpha\beta s + \beta^2 t)^2 = f(\alpha,\beta).
\]

More precisely, $\calA_\sigma$ is the associative $K$-algebra generated by
$r,s,t$ with relations
\[
r^2 = a,\qquad rs + sr = b,\qquad rt + tr + s^2 = c,\qquad
st + ts = d,\qquad t^2 = e.
\]
These relations ensure that the above quadratic identity holds
identically in $\alpha, \beta$.

A key feature of this construction is the existence of a natural
central element
\[
\xi = s^2 - c \in Z(\calA_\sigma).
\]
Working over the polynomial ring $K[\xi]$, the algebra $\calA_f$ may be
identified with the Clifford algebra associated to a ternary quadratic
form
\[
Q_\xi(x,y,z) = ax^2 + bxy + cy^2 + dyz + ez^2 + \xi(y^2 - xz).
\]

Geometrically, this quadratic form arises from the standard embedding
of the curve $C_f : y^2 = f(x,z)$ into $\P^3$, where the image
is cut out by a pencil of quadrics whose generic member is given by
\[
x_4^2 = Q_\xi(x_1,x_2,x_3).
\]
This interpretation links the algebra $\calA_\sigma$ directly to the geometry
of the genus one curve defined by $f$.

\noindent
\textbf{Explicit construction of Azumaya algebras for $\bf{n=3}$.} Let $0 \ne \sigma \in \Sha(E)[3]$, and let
\[
f(x,y,z)=ax^3+by^3+cz^3+a_2x^2y+a_3x^2z+b_1xy^2+b_3y^2z
+c_1xz^2+c_2yz^2+mxyz
\]
be a ternary cubic defined over $K$ associated with $\sigma$.

Following Kuo \cite{Kuo11}, one associates to $\sigma$ a noncommutative $K$-algebra $\calA_\sigma$
defined as follows.
Assuming $c\neq 0$, the algebra $\calA_\sigma$ is generated by two elements $x$ and $y$,
subject to relations obtained from the formal identity
\[
f(\alpha,\beta,\alpha x+\beta y)=0
\]
in the variables $\alpha$ and $\beta$.
Equivalently, $\calA_\sigma$ is the quotient of the free associative algebra $K\{x,y\}$
by the ideal generated by the four cubic relations arising from equating
coefficients in this identity.
Although the defining relations appear asymmetric, the resulting algebra
depends functorially on the ternary cubic $f$ and behaves well under changes
of coordinates.

The algebra $\calA_\sigma$ is an Azumaya algebra over its center, and its center is canonically isomorphic to the coordinate ring of $E\setminus\{0\}$.
In particular, after base change to the function field $K(E)$,
the algebra $\calA_\sigma \otimes_Z K(E)$ becomes a central simple algebra over $K(E)$,
and hence defines a Brauer class on $E$.

This construction admits a concrete cohomological interpretation.
Let $L/K$ be a cyclic extension of degree $3$ over which the plane cubic curve $C_f : f = 0$
acquires a rational point.
After possibly replacing $K$ by a bi-quadratic extension, one may assume that $K$ contains a primitive cube root of unity $\zeta_3$ and that
$L=K(\sqrt[3]{a})$.
In this setting, the algebra $\calA_\sigma\otimes_Z K(E)$ may be described explicitly
as a cyclic algebra $(L/K,\, g)$,
where $g\in K(E)^\times$ is a rational function whose divisor is
\begin{equation}\label{eq:de_jong_3_divisor}
\operatorname{div}(g)
=(P)+(\tau P)+(\tau^2 P)-3(0),
\end{equation}
for a suitable point $P\in E(L)$ arising from the Galois action of $\Gal(L|K) = \langle \tau \rangle$ on points in $C_f(L)$.

\subsection{Restriction of scalars}

Agashe and Stein \cite[Proposition 2.4]{AgasheStein} used restriction of scalars to prove that every element of the Shafarevich--Tate group of an elliptic curve of order $n$ can be visualized inside an abelian variety of dimension $n$. In this subsection, we briefly recall the construction.

Let $E$ be an elliptic curve over a number field $K$, and let
$c \in \Sha(E)$ be an element of order $n$. A key input is a classical
result of Cassels \cite{Cassels62}, which asserts that
there exists a finite extension $L/K$ with $[L:K]=n$ such that $\operatorname{res}_L(c)=0 \in H^1(L,E)$. An alternative proof was later given by
O'Neil \cite{Oneil2002}.

Once such an extension is fixed, one applies the restriction of scalars
construction. Let $J \coloneqq \operatorname{Res}_{L/K}(E_L)$,
which is an abelian variety over $K$ of dimension $n$. The functorial
properties of Weil restriction yield a natural closed immersion
$E \hookrightarrow J$. Moreover, by Shapiro's lemma, there is a canonical
identification
\[
H^1(K,J) \cong H^1(L,E),
\]
under which the image of $c$ is identified with $\operatorname{res}_L(c)$.
Since $\operatorname{res}_L(c)=0$, the class $c$ lies in the kernel of the map
$H^1(K,E) \to H^1(K,J)$, and hence is visible in $J$. This shows that any element of $\Sha(E)$ of order $n$ is visualized in an
abelian variety of dimension $n$.

\subsection{Non-emptiness of the Visibility categories}\label{subsec:vis_cats_nonempty}

The constructions of de Jong and restriction of scalars ensure that the $\sigma$-visibility category $\mathcal{V}(A/K,\sigma)$ has a nonempty set of objects. From this, it easily follows that for every $n \ge 2$, the visibility category
$\mathcal{V}(E/K,n)$ has a nonempty set of objects. Indeed, let $\sigma_1,\sigma_2 \in \Sha(E/K)[n]$ and choose
$B_i \in \mathcal{V}(E/K,\sigma_i)$ for $i=1,2$. Consider the abelian variety $B$
obtained by gluing $B_1$ and $B_2$ along their common abelian subvariety $E$. Then the subgroup $\langle \sigma_1,\sigma_2 \rangle \subset \Sha(E)$ is contained in the kernel of the map $\Sha(E) \longrightarrow \Sha(B)$. Since $\Sha(E/K)[n]$ is finite, the result follows.

\subsection{Reinterpreting finiteness of \texorpdfstring{$\Sha(E/K)$}{Sha} in terms of visibility}

An interesting reframing of the finiteness of $\Sha(E)$ follows readily from existing results, which we feel is of value to write down explicitly. We do this in the more general context of an abelian variety. 

\begin{proposition}
    Let $A$ be an abelian variety over a number field $K$. Then $\Sha(A/K)$ is finite if and only if there exists an abelian variety $B$ in which every element of $\Sha(A/K)$ becomes visible.
\end{proposition}

\begin{proof}
    If $\Sha(A/K)$ is finite, then by the argument in \Cref{subsec:vis_cats_nonempty}, we obtain an abelian variety $B$ that visualizes all of $\Sha(A/K)$. Conversely, if such a $B$ exists, then we have $\Sha(A/K) =  \Vis_B(\Sha(A/K)) \subset \Vis_B(H^1(K, A))$, and this latter group is finite, since it fits into the exact sequence \[ 0 \to B(K)/A(K) \to C(K) \to \Vis_B(H^1(K, A)) \to 0 \] (where $C := B/A$ is the quotient abelian variety) and is thus simultaneously torsion and the surjective image of a finitely generated group.
\end{proof}

\subsection{Auxiliary Lemmas on visibility}\label{subsec:main_setup}

In this section we review some basic facts about visibility. We start with the following straightforward observations.

\begin{observation}\label{obs1}
Suppose that we have a tower of abelian varieties $E\subset A\subset B$. Then, we have the following maps:
\[
H^1(K,E)\longrightarrow H^1(K,A)\longrightarrow H^1(K,B).
\]
This implies that 
$\Vis_{A} H^1
(K, E)\subset \Vis_{B} H^1
(K, E).$ 
We have equality
\begin{align*}
    \Vis_{A} H^1(K, E)= \Vis_{B} H^1(K, E)
\end{align*}
whenever $H^1(K,A)\longrightarrow H^1(K,B)$ is injective, or equivalently, when $\Vis_{B} H^1
(K, A)=(0)$.    
\end{observation}

The following results are taken from \cite[Section 2]{Fisher_vis7} which will be used throughout the paper.

Let $E$ and $F$ be abelian varieties over a number field $K$, with common finite Galois
submodule $\Delta$ of size $d$. Let $A = (E \times F)/\Delta$, where the quotient is by the diagonal embedding of $\Delta$.
Let $F'= A/E$ and $E' = A/F$. There is then a commutative diagram
\begin{align}\label{diagram}
\xymatrix{&&0\ar[d]&&\\
&&F\ar[d]\ar^{\psi}[dr]&&\\
0\ar[r]& E\ar_{\phi}[dr]\ar[r]&A\ar[d]\ar[r]&F'\ar[r]&0\\
&&E'\ar[d]&&\\
&&0&&}    
\end{align}
where the row and column are exact sequences of abelian varieties, and the diagonal maps $\phi$
and $\psi$ are isogenies with kernel $\Delta$.

The following lemma gives a group containing $\Vis_A\Sha(E)$ when $E'(K)/\phi E(K) = 0$, which happens for example when $E$ has rank $0$ and $\#E(K)$ is relatively prime to $\#\Delta$.

\begin{lemma}[{\cite[Lemma 1]{Fisher_vis7}}]
\label{lemma:fisher_2.1}
If $E'(K)/\phi E(K) = 0$ then $F'(K)/\psi F(K) \cong \Vis_A H^1(K, E)$.
\end{lemma}

The following theorem implies that the inclusion $ \Vis_A \Sha
(E)\to F'(K)/\psi F(K)$ is actually an isomorphism when we assume certain local conditions.

\begin{theorem}[{\cite[Theorem 3.1]{AgasheStein}, \cite[Theorem 2.2]{Fisher_vis7}}]\label{thm:fisher_2.2}
If $E'(K)/\phi E(K) = 0$ and
\begin{itemize}
    \item [(i)] all the Tamagawa numbers of $E$ and $F'$ are coprime to $d$, and
    \item [(ii)] $A$ has good reduction at all places $v\mid d$, and
    \item [(iii)] writing $p$ for the rational prime below $v$, $e(K_v/\Q_p) < p-1$ for all places $v\mid d$,
\end{itemize}
then $F'(K)/\psi F(K) \cong \Vis_A \Sha(K, E)$.
\end{theorem}

\begin{remark}
    This version of \cite[Theorem 3.1]{AgasheStein} is taken from \cite[Theorem 4.1]{frengley2026modular} and accounts for a typographical error in Fisher's version. Note that we have omitted the explicit condition on $d$ being odd since this is implied by hypothesis (iii).
\end{remark}

\section{Minimal visualizations of different dimensions}\label{sec:different_minima}

In this section we prove \Cref{thm:different_minima}; i.e. we show that the elliptic curve \[ E : y^2=x^3-x^2-169321x-28379327, \] which can be found as curve \href{https://www.lmfdb.org/EllipticCurve/Q/161472/bz/1}{\texttt{161472.bz.1}} in the LMFDB, admits minimal visualizations of $\Sha(E)[5]$ of dimensions $2$ and $3$, which are given explicitly as in the statement.

Rather than proceeding with the mere verification that the abelian surface $A$ and abelian threefold $B$ are visualizations of $\Sha(E/\Q)[5]$, we instead present the more interesting details of how we arrived at this particular elliptic curve, thereby providing a commentary to the Magma file \texttt{magma/g2RMapproach.m}.

We start with the question of how to find a nontrivial example of an elliptic curve that admits a $3$-dimensional visualization of all elements of $\Sha(E/\Q)[5]$. The basic approach here is due to Fisher and can be found in \cite[Section 5]{Fisher_vis7}. (Fisher deals with order $7$ elements there but the basic idea applies for any odd prime.)

\subsection*{Step 1. Obtain list of candidate elliptic curves.} 

Clearly we want an elliptic curve $E$ with $\Sha(E/\Q)[5] \neq 0$. However, since we will want to apply \Cref{thm:fisher_2.2}, we also want $E$ to have good reduction at $5$, we want all Tamagawa numbers to be coprime to $5$, and we want $E$ to only have $5$-part of $\Sha$. These curves were downloaded from the LMFDB and are available in the file \texttt{data/ellipticcurvesinfo5.txt}. Each row corresponds to a curve satisfying the above requirements, in the format \texttt{aInvs:cond:label:aps}, where \texttt{aInvs} are the Weierstrass coefficients of the curve, \texttt{cond} is the conductor, \texttt{label} is the LMFDB label, and \texttt{aps} is a list of the trace of Frobenius values $a_p$ for primes $p < 100$.    

\subsection*{Step 2. Obtain list of genus $2$ curves whose Jacobians can be glued with some candidate elliptic curve to obtain $3$-dimensional visualizations.}

The $3$-dimensional visualization we are attempting to construct will arise from the setup of \Cref{subsec:main_setup}; in particular, we will want to glue $E[5]$ to a Galois submodule of the $5$-torsion of the Jacobian of a genus $2$ curve. One way to obtain this is to work exclusively with genus $2$ curves $C$ that admit real multiplication by $\sqrt{5}$, for then $J_C[\sqrt{5}]$ would at least have the correct structure as an abelian group to be able to be isomorphic to $E[5]$ as Galois modules for $E$ an elliptic curve.

To generate these genus $2$ curves, we use \cite[Corollary 7.1]{cowan2024generic} that provides a sextic polynomial $f_5(a,b,c; x)$ over $\Q[a,b,c]$ such that, for generic choices of $a,b,c \in \Q$, the genus $2$ curve $y^2 = f_5(a,b,c; x)$ has RM by $\sqrt{5}$. This polynomial can be found in our repository as \texttt{data/5.txt}.

We specialize $f_5(a,b,c; x)$ at low-height rationals $a,b,c$. For each genus $2$ curve arising in this way, we check if there is an elliptic curve $E$ in the list obtained in Step 1 such that there is a $5$-congruence $J_C[\sqrt{5}] \cong E[5]$. To achieve this, we define, for $p$ a prime, the following quantities:
\begin{align*}
    n_1 &:= \#C(\F_p) & t_p &:= p + 1 - n_1\\
    n_2 &:= \#C(\F_{p^2}) & n_p &:= (n_1^2 + n_2)/2 - (p+1)n_1 - p.
\end{align*}
We then check, for each of the 25 $a_p(E)$ values in the \texttt{aps} column, whether $a_p(E)$ satisfies $x^2 \pm t_px + n_p \equiv 0 \Mod{5}$ for some choice of sign $\pm$. If this is indeed satisfied for some $E$ as in Step $1$, we make two further checks on $C$: (1) it should have good reduction at $5$; (2) its Jacobian should have rank $4$ and no torsion supported at $5$. Condition (1) is needed to ensure \Cref{thm:fisher_2.2} can be applied, and condition $2$ will give that, for $B := (J_C \times E)/\Delta$ for $J_C[\sqrt{5}] \cong \Delta \cong E[5]$, we have $\Vis_B(\Sha(E/\Q)) \cong \Sha(E/\Q) \cong (\Z/5\Z)^2$.

If the pair $(C,E)$ survives all of these conditions, we append it to a running set of \texttt{goodpairs}. Of course, merely checking $25$ $a_p$ values does not guarantee a $5$-congruence; for the pairs in \texttt{goodpairs}, we gain extra confidence by extending the computation for all primes $p < 1000$; pairs that survive this subsequent filtering are added to another running set \texttt{verygoodpairs}. For the pairs in this set, it is very likely that $E[5] \cong J_C[\sqrt{5}]$ as Galois modules up to a possible quadratic twist.

At this point, we take the first pair in this set, which corresponds to the $E$ as in the statement of \Cref{thm:different_minima}, and the curve \[ C : y^2 = x^6 + 2x^4 + 12x^3 + 5x^2 + 6x + 1,\] and we seek to identify the correct quadratic twist. Replacing the check on $a_p(E)$ satisfying $x^2 \pm t_px + n_p \equiv 0 \Mod{5}$ with $a_p(E_d)$ satisfying $x^2 - t_px + n_p \equiv 0 \Mod{5}$ for various choices of $d$ indicates that $E[5] \cong J_{C_{-2}}[\sqrt{5}]$. This can be certified in the same way as described in \cite[Section 6]{Fisher_vis7}, by using the analytic representation of the $\sqrt{5}$-endomorphism, and algebraizing the kernel. (See also \cite[Section 6.4]{frengley2024explicit}.) This explains the $-2$ in the statement of \Cref{thm:different_minima}. Henceforth in this proof, $C$ will denote $C_{-2}$ (i.e. the one written in the statement of the theorem we are currently proving). This is how we identified the curve $E$ and $C$ as in the statement of the theorem.

\subsection*{Step 3. Verify Tamagawa number condition}

Since we have heavily relied upon \Cref{thm:fisher_2.2}, we need to ensure that the Tamagawa numbers of $J_C$ are coprime to $5$. This is achieved by using van Bommel's \texttt{Tamagawa} package \cite{bommel2021efficient} and is carried out in \texttt{magma/verifyThm1.5.m}.

Finally, notice that $B$ is minimal in $\mathcal{V}(E, 5)$ because the Jacobian of $C$ is a simple abelian variety. So $B$ doesn't contain proper subabelian varieties containing $E$. This yields \Cref{item:3d_miminal_vis} of \Cref{thm:different_minima} regarding the $3$-dimensional minimal visualization.

\subsection*{Step 4. Find $2$-dimensional visualization}

For \Cref{item:2d_miminal_vis} of \Cref{thm:different_minima}, the idea is relatively straightforward and similar to above: having $E$ in hand, we want to find a $5$-congruent curve $F$ for which $E$ and $F$ satisfy the conditions of \Cref{thm:fisher_2.2} and such that $\Vis_A(\Sha(E/\Q)) \cong (\Z/5\Z)^2$. This is easier to obtain than the $3$-dimensional visualization since the modular curve $X_E(5)$ has genus 0 and so there are infinitely many $5$-congruent curves; it is just a matter of finding one that satisfies the additional required conditions. Moreover, finding $5$-congruent elliptic curves is conveniently already implemented in Magma as \texttt{HessePolynomials} and \texttt{RubinSilverbergPolynomials} (Rubin and Silverberg gave the first parametrization of $X_E(5)$ in \cite{rubin1995families}, and Fisher gave parametrizations for both symplectic and anti-symplectic congruences using the Hessian in \cite{Fisher_hess12}). The entire implementation that allowed us to find the curve $F$ as in the statement of the theorem is in \texttt{magma/findCongruentCurve.m}, which yields the following result that finishes the proof of \Cref{thm:different_minima}.

\begin{proposition}
    There is an abelian surface $A$ over $\Q$ that visualizes $\Sha(E)[5]$, and we have \[ A \cong (E \times F) / \Delta, \] where $F$ is the elliptic curve \[ y^2 = x^3 - x^2 - 2285384981x - 42050896792563, \] and $E[5] \cong \Delta \cong F[5]$, diagonally embedded in each component of $E \times F$.
\end{proposition}

\begin{proof}
    Let us consider the abelian varieties $E' := A/F$, and $F' := A/E$ and the isogenies $\psi:F\to F'$ and $\phi:E\to E'$. By construction, the map $\phi$ has kernel $E[5]$, and hence $E'$ is isogenous to $E$ over $\Q$. Since $E$ is the only curve in its isogeny class, we obtain that $E'$ is equal to $E$, and therefore $\phi$ is multiplication by $5$. We similarly obtain $F = F'$, and $\psi$ is multiplication by $5$.

    We seek to apply \Cref{thm:fisher_2.2}, for which we need $E(\Q)/5E(\Q) = 0$; however this is true because $E(\Q)$ is trivial. The Tamagawa products of $E$ and $F$ are $2$ and $4$ respectively, as is readily checked in Sage. That both curves have good reduction above $5$ just follows from their conductors being coprime to $5$, also easily checked in Sage. We obtain \[ \Vis_A\Sha(E/\Q) \cong F(\Q)/5F(\Q). \] Computing the Mordell-Weil group of $F$ in Magma yields \[ \Vis_A\Sha(E/\Q) \cong (\Z/5\Z)^2 \cong \Sha(E)[5];\] i.e. $A$ visualizes all of $\Sha(E)[5]$ and hence $A$ belongs to $\mathcal{V}(E, 5)$. The minimality of $A$ is clear since it has dimension $2$.
\end{proof}

\section{Minimality of the restriction of scalars construction}\label{sec:restiction_minimality}

In this section we prove \Cref{thm:restrictionofscalars_minimal_generaln}. We begin with the following preliminary result.

\begin{lemma}\label{lemma:Sn extension}
Let $\sigma \in \Sha(E/K)$ be an element of order $n \ge 2$.
Suppose there exists an extension $L/K$ of degree $n$ such that
$C_{\sigma}(L) \neq \varnothing$, and that the normal closure $M$ of $L/K$
has Galois group $\Gal(M/K) \cong S_n$.
Then the Weil restriction $A \coloneqq \Res_{L/K}(E_L)$
is a minimal object in the visibility category
$\mathcal{V}(E/K,\sigma)$.
More generally, this is true if $\Gal(M/K)$ is a doubly transitive subgroup of $S_n$.
\end{lemma}

\begin{proof}
    Let $M$ be the normal closure of $L/K$, set $G \coloneqq \Gal(M/K)$, and $H \coloneqq \Gal(M/L)$. By assumption $G \cong S_n$ and $H$ is the stabilizer of a point for the natural action of $S_n$ on $\{1,\dots,n\}$, hence $H \cong S_{n-1}$.

Let $A=\Res_{L/K}(E_L)$. After base change to $M$, the defining property of
Weil restriction gives a canonical $M$-isomorphism
\[
A_M \cong \prod_{\tau \in G/H} E_M,
\]
and the action of $G$ on $A_M$ permutes the factors through its left action on
the coset set $G/H$.

Let $W \coloneqq \Ind_H^G(\mathbf{1}_H)$
be the induced representation of $G$ of the trivial representation $\mathbf{1}_H$ of $H$. The $G$-action on the factors of $A_M$ shows that the $K$-isogeny
decomposition of $A$ is governed by the decomposition of the $\Q[G]$-module $W$, that is,
each irreducible component of $W$ corresponds to a $K$-isotypic isogeny factor of $A$.

Since $G \cong S_n$ and $H \cong S_{n-1}$, the representation $W$ is the usual
permutation representation of $S_n$ on $n$ letters. Hence $W \cong \mathbf{1}_G \oplus V$,
where $\mathbf{1}_G$ is the trivial representation and $V$ is the standard
representation on $G$ of dimension $n-1$, which is irreducible over $\Q$. Therefore,
up to $K$-isogeny,
\[
A \sim_K E \times B,
\]
where the factor $E$ corresponds to the trivial summand and the complementary
factor $B$ corresponds to $V$. Since $V$ is $\Q$-irreducible, the abelian
variety $B$ is $K$-simple.

If $G$ is any subgroup of $S_n$ and $H$ is the stabilizer in $G$ of a point, we still have a decomposition $\Ind_H^G(\mathbf{1}_H) \cong \mathbf{1}_G \oplus V|_G$. Here $V|_G$ is the restriction of the standard $(n-1)$-dimensional representation $V$ of $S_n$ to the subgroup $G$, which is known to be irreducible if and only if $G$ is doubly transitive.

Finally, the hypothesis $C_\sigma(L)\neq\varnothing$ means that $\sigma$
becomes trivial over $L$, so $\sigma$ is visible in $A=\Res_{L/K}(E_L)$; in
particular $A$ is an object of $\mathcal{V}(E/K,\sigma)$. Because the
complement of $E$ in $A$ is $K$-simple, there is no proper abelian subvariety
$A'\subsetneq A$ with $E\subset A'$ in which $\sigma$ is visible. Hence $A$ is
minimal in $\mathcal{V}(E/K,\sigma)$.
\end{proof}

The previous lemma reduces the proof of \Cref{thm:restrictionofscalars_minimal_generaln}
to the construction of a degree-$n$ extension $L/K$ such that
$C_\sigma(L)\neq\varnothing$ and whose normal closure has Galois group $S_n$
(or, more generally, is doubly transitive).
We now construct such an extension using a geometric argument based on
the degree-$n$ embedding of the torsor $C_\sigma$ and Hilbert irreducibility.

\begin{proof}[Proof of \Cref{thm:restrictionofscalars_minimal_generaln}]
    Let $C_{\sigma}$ be a genus-one torsor corresponding to $\sigma$. Then $C_{\sigma}$ has a model as a smooth degree $n$ curve in $\P^{n-1}$. Let $\Pdual^{n-1}$ denote the dual space parametrizing hyperplanes in $\P^{n-1}$. Let $Y$ be a variety parametrizing pairs consisting of a hyperplane $H \subset \P^{n-1}$ and an ordered $n$-tuple of points $\langle P_1,P_2,\ldots,P_n \rangle$, where $P_1+P_2+\ldots+P_n$ is the divisor on $C_{\sigma}$ obtained by intersection with $H$. Note that $Y \simeq C_{\sigma}^{n-1}$ since the $n-1$ points $P_i, 1 \leq i < n$ determine the hyperplane $H$ and also the last point of intersection $P_n$. Forgetting the ordered $n$-tuple of points $P_i$ gives rise to a branched covering map $Y \rightarrow \Pdual^{n-1}$. The symmetric group $S_n$ acts naturally on this cover, by permuting the points in the ordered $n$-tuple.
    Let $C_\sigma^\vee \subset (\P^{n-1})^\vee$ be the dual variety of $C_\sigma$, that is, the Zariski closure of the set of all tangent hyperplanes to $C_\sigma$. Since the dimension of $C_\sigma^\vee$ is at most $n-2$ (see \cite[Section 1.1]{Tev03}), it is a proper closed subvariety of the dual projective space $(\P^{n-1})^\vee$. In particular, a general hyperplane in $\P^{n-1}$ meets $C_\sigma$ transversely. The associated monodromy action on these points is the full symmetric group $S_n$, so the geometric Galois group of this cover is $S_n$.
    By Hilbert irreducibility, there exists a hyperplane $H$ defined over $K$, such that the permutation action on the points of intersection $H \cap C_\sigma = \{P_i, 1 \leq i \leq n\}$ has maximal image $S_n$. Let $L = K(P_1)$ be the field of definition of the point $P_1$. Then $L|K$ is a degree $n$ extension of $K$, with Galois group of normal closure $= S_n$, and $C_\sigma(L) \ne \varnothing$. By \Cref{lemma:Sn extension}, the $n$-dimensional abelian variety $A \coloneqq \Res_{L/K}(E_L)$ is minimal in the visibility category $\mathcal{V}(E/K,\sigma)$.
\end{proof}


\section{Explicit de Jong construction for orders $2$ and $3$}\label{sec:explicit_cremona_mazur_for_n_2_3}

In this section, we study the minimality of elements in the category $\V(E/K,\sigma)$ arising from the de Jong construction.

\subsection{Minimal elements via de Jong construction: case $\bf{n=2}$}
\label{subsec:explicit_cremona_mazur_for_n_2}
In this subsection, we prove \Cref{thm:explicit_cremona_mazur_for_n_2}.
\begin{proof}[Proof of \Cref{thm:explicit_cremona_mazur_for_n_2}]
Consider the quaternion algebra $\mathcal{A}_\sigma$ associated to $\sigma\in\Sha(E)[2]$, given in \cite[Section 2.2]{Fisher_alge19}, originally due to Haile and Han \cite{HaHa07}. Let $g \in K[x, z]$ be the binary quartic associated to $\sigma$, say 
\[
g(x,z)=ax^4 +bx^3z+cx^2z^2 +dxz^3 +ez^4.
\]
We may describe $\mathcal{A}$ as the Clifford algebra of the quadratic form
\[
Q_{\xi}(x,y,z)=ax^2 +bxy+cy^2 +dyz+ez^2 +\xi(y^2 -xz),
\]
where $\xi$ is the $x$-coordinate on $E$. That is, $\calA_\sigma$ is the associative algebra generated by $u_1, u_2, u_3$ subject to the relations coming from the identity $(x u_1 + y u_2 + z u_3)^2 = Q_\xi(x,y,z)$. It is known that $\calA_\sigma \simeq \calA_+ \otimes K(\eta)$, where $\calA_+$ is the even subalgebra and $\eta$ is the $y$-coordinate on $E$ (see \cite[Sections 2.1, 3.1]{Fisher_alge19}). The even sub-algebra $\calA_+$ is a quaternion algebra over $K[\xi]$, and is generated as a vector space by $1, u_1u_2, u_2u_3, u_1u_3$. It is easy to check that $(b-2u_1u_2)^2 = \Delta - 4a \xi$, where $\Delta = b^2-4ac$.
So $\mathcal{A}_\sigma$ contains the quadratic extension $L=K(E)(\sqrt{\Delta - 4a \xi})$ of $K(E)$.
By \cite[Lemma 5.4.7]{Voight}, the algebra $\mathcal{A}_\sigma \otimes L$ splits over $L$.

For any curve $X$ over $K$, we have an exact sequence of Galois modules
    $0 \rightarrow \Pic^0(X) \rightarrow \Pic(X) \rightarrow \Z \rightarrow 0$.
If $X$ has a $K$-rational point or a $K$-rational divisor of degree $1$, the sequence splits, and we get that $H^1(K,\Pic X) \simeq H^1(K,\Pic^0 X) = H^1(K, \Jac(X))$. A map $X \rightarrow E$ leads to a commutative diagram
\begin{align*}
    \xymatrix{
    0\ar[r] & \Br(K) \ar_{=}[d]\ar[r] & \Br(E) \ar[d]\ar[r] & H^1(K,\Pic E) \ar[d]\ar[r] & 0\\
    0\ar[r] & \Br(K) \ar[r] & \Br(X) \ar[r] & H^1(K,\Pic X) \ar[r] & 0.}
\end{align*}
If $X(K) \neq \varnothing$, the last vertical map is exactly the push-forward map $H^1(K,E) \rightarrow H^1(K,\Jac X)$ of our interest.
Note that we can change the class of $\calA_\sigma \in \Br(E)$ by any constant, i.e., by a class coming from $\Br(K)$, without affecting the associated torsors.
Since the quaternion algebras
\begin{align*}
    \calA_\sigma = \frac{(a,\Delta-4a\xi)}{K(E)}, \qquad \calA' = \frac{(a,\frac{\Delta-4a\xi}{-4a})}{K(E)} = \frac{(a,\xi - \frac{\Delta}{4a})}{K(E)}
\end{align*}
differ by a constant Brauer class, we deduce that $\calA'$ also represents $\sigma \in \Sha(E)[2] \subseteq H^1(K,E)$.

Let $L' = K(E)\left(\sqrt{\xi - \frac{\Delta}{4a}}\right)$ and $X$ denote the curve with function field isomorphic to $L$.
With $f$ as in the statement of \Cref{thm:explicit_cremona_mazur_for_n_2}, it is easy to see that the equation
\[
y^2=f\left(\frac{b^2 - 4a(c - x^2)}{4a} \right)
\]
gives a model of the genus $2$ curve $X$.
It is clear that the map $X \rightarrow E$ is totally ramified at the two points in $E$ with $\xi$-coordinate equal to $\frac{\Delta}{4a}$. Furthermore, since the hyperelliptic polynomial defining $X$ is monic, the two points in $X$ above $\infty$ are rational, i.e, $X(K) \ne \varnothing$. So the element $\sigma$ is visualized in the Jacobian $\Jac(X)$.
In particular, this construction gives a minimal element in the visualization category $\V(E; \sigma)$. By \cite[Theorem 14.1.1]{cassels_flynn_1996}, we deduce that $\Jac(X)$ is isomorphic to $(E \times F)/\Delta$, where $F$ is the elliptic curve given by the equation
\begin{align*}
    F : y^2 = (8a^2d - 4abc + b^3)^2 x^3 - (64a^3e - 16a^2bd - 16a^2c^2 + 16ab^2c - 3b^4) x^2 + (3b^2-8ac)x + 1
\end{align*}
and $\Delta = E[2] \simeq F[2]$ is embedded diagonally in $E \times F$.

If $E(K)/2E(K) = \{0\}$, then $E(K)[2] = \{0\}$ and hence $F(K)[2] = \{0\}$. If $\sigma \in \Sha(E)[2]$ is non-trivial, then \Cref{lemma:fisher_2.1} immediately lets us deduce that the Mordell-Weil rank of $F$ is at least $1$, since $F(K)/2F(K) \simeq \Vis_{\Jac(X)}(H^1(K,E)) \ne \{0\}$. 
\end{proof}

\subsection{Minimal elements via de Jong construction: case $\bf{n=3}$}
\label{subsec:explicit_cremona_mazur_for_n_3}
Let $0\ne\sigma\in \Sha(E/K)$ and let
\[
f(x,y,z)=ax^3+by^3+cz^3+a_2x^2y+a_3x^2z+b_1xy^2+b_3y^2z
+c_1xz^2+c_2yz^2+mxyz
\]
be a ternary cubic defined over $K$ associated with $\sigma$.
We assume that $K$ is a number field that contains the cube root of unity $\zeta_3$ and $\sqrt{d}$, where $d$ is the discriminat of the cubic polynomial $f(x,0,1)$. Consequently, we can assume that $f$ has the form
\[
f(x,y,z)=ax^3+by^3-z^3+b_1xy^2+b_3y^2z+
mxyz.
\]
Let $L=K(\sqrt[3]{a})$ and $\Gal(L|K) = \langle \tau \rangle$. Notice that $L$ is a cubic extension of $K$, as otherwise $f$ would have a root in $K$, contradicting that $\sigma\ne 0$. By \cite[Section 6.4]{Fisher_alge19}, the element
\begin{equation}\label{cubic element}
g(\xi,\eta) := \eta-\zeta_3^2m\xi-3(1-\zeta_3)ab+\frac{1}{9}(\zeta_3-\zeta_3^2)m^3
\end{equation}
is a cubic element in $\mathcal{A}_\sigma$, and as in \Cref{eq:de_jong_3_divisor}, we have  $\divi (g) = (P) + (\tau P) + (\tau^2 P) - 3(0)$ for a certain point $P \in E(L)$ with $x$-coordinate $-( 1 / 3 ) m^2 + b_1 \sqrt[3]{a} - b_3 (\sqrt[3]{a})^2$. Consider the cubic extension $M=K(E)[z]/(z^3-g(\xi,\eta))$ of $K(E)$, and denote by $X$ the associated degree $3$ cover of $E$. Notice that $\tau(P)=P$ if and only if $b_1=b_3=0$. If this condition does not hold, then the map $X\to E$ is totally ramified at $P$, whence $X$ is a de Jong curve of genus $4$.

  \begin{proof}[Proof of \Cref{thm:minimality_CremonaMazur_order3}]                      
  Fisher's database \cite{Fisher_Database_Order3Sha} provides, for each elliptic curve $E/\Q$ of conductor less than $300{,}000$, a list of ternary cubics representing the non-zero 
  elements of $\Sha(E)[3]$ up to sign. Restricting to curves      
  without a rational $3$-isogeny leaves $90{,}932$ such ternary   
  cubics. For each one, the function
  \texttt{CMcurveForIndex3Torsor} defined in
  \texttt{magma/Minimality\_CremonaMazur\_index3.m} constructs the
   de Jong curve $X$ of genus $4$ over an at-most biquadratic
  number field $K$
  using the cubic extension $M=K(E)[z]/(z^3-g(\xi,\eta))$ of $K(E)$ described above.
  We then search
  among primes less than $1000$ that are totally split in $K$, for
   a prime $p$ such that the corresponding $L$-polynomial
  $L_p(X,t)$ has an irreducible factor of degree $6$. Finding such
   a prime certifies that there are no non-trivial sub-abelian
  varieties of $\Jac(X)$ containing $E$, and thus $\Jac(X)$ is a
  minimal visualizing object. For every one of the $90{,}932$
  ternary cubics, we found such a prime witness. The witnesses are
   recorded in the \texttt{data/} directory: the original files
  from Fisher's website live in \texttt{sha\_order3\_raw/}, their
  restriction to curves without a rational $3$-isogeny in
  \texttt{sha\_order3\_no3\_iso/}, and the same files annotated
  with the prime witnesses in
  \texttt{sha\_order3\_no3\_iso\_witnessed/}. Intermediate
  per-torsor files consumed by the Magma computation sit in the
  two \texttt{sha\_order3\_processed} variants.
  \end{proof}

  \begin{remark}
      The full computation is reproducible from a clean checkout
  via the script \texttt{run\_pipeline.sh} at the repository root.
   It (i) downloads the raw data from Fisher's website, (ii)
  filters to curves without a rational $3$-isogeny using the      
  Cremona labels in \texttt{data/cremona\_labels\_no\_3\_iso.txt},
   (iii) processes the filtered files into per-torsor input, (iv)
  runs \texttt{magma/PrimeWitnesses.m} in parallel over the
  processed files, and (v) merges the resulting prime witnesses
  back into the files in
  \texttt{data/sha\_order3\_no3\_iso\_witnessed/}. Running the
  pipeline requires Python~3, GNU~\texttt{parallel}, and Magma.
  \end{remark}

\begin{remark}                                   The restriction to curves $E$ without a rational $3$-isogeny is essential for our approach to work.  When such an isogeny $\phi\colon E\to E'$ exists, the de Jong Jacobian $\Jac(X)$ contains $E'$ as a subvariety, so the characteristic polynomial of Frobenius always factors through degree at most~$4$ and no prime witness can exist. 
\end{remark}

\begin{remark}
    \label{rem:diagonal_torsors}
    Let $E$ be an elliptic curve over any number field $K$, and let $0 \ne \sigma \in H^1(K,E)[3]$ be represented by a ternary cubic $a x^3 + b y^3 - z^3 + m xyz$. Then we deduce that the divisor of the function $g \in K(\zeta_3)(E)$ given in \Cref{cubic element} is $3 (P) - 3 (0)$, where $P = (-\frac{m^2}{3}, 3(1-\zeta_3)ab + \frac{1}{9}(\zeta_3 + 2)m^3)$. This shows that $P$ is a $3$-torsion point in $E$, and the subgroup generated by $P$ is isomorphic to $\mu_3$. Moreover, the associated map $C \rightarrow E$ is unramified, so that the genus of $C$ is also $1$. Hence, the induced map on Jacobians is an isogeny $\phi: E \rightarrow \Jac(C)$ of degree $3$ with kernel generated by $P$, and the torsor $\sigma$ is visualized by the isogeny $\phi$. In other words, $\sigma$ is obtained from a class in $H^1(K,\ker \phi) \simeq H^1(K,\mu_3) \simeq K^{\times}/{K^\times}^3$ by the natural map $H^1(K,\ker \phi) \rightarrow H^1(K,E)$.
\end{remark}

\section{de Jong curves need not yield minimal visualizations, and the difference between $\V(E,\sigma)$ and $\V(E,\ord(\sigma))$}\label{sec:cremona_mazur_curves_can_have_large_genus}

In this section we establish \Cref{thm:cremona_mazur_curves_can_have_large_genus,thm:categories_not_equal}.

Let $E$ be the elliptic curve \href{https://www.lmfdb.org/EllipticCurve/Q/571/b/1}{571.b1}. We have $E(\Q)=\{0\}$, $\Sha(E/\Q) \cong (\Z/2\Z)^2$, with the elements of $\Sha(E/\Q)$ represented by the binary quartics $\psi_i, 1 \leq i \leq 4$ shown in \Cref{eq:binary_quartics_example}.
Let us denote by $\calA_{\psi_i}$ the Azumaya algebra of rank $4$ over $\Q(E)$ associated with $\psi_i$. Let $C_i$ denote the associated hyperelliptic curve of genus $2$, given in \Cref{thm:explicit_cremona_mazur_for_n_2}. Then we have the following descriptions of $\calA_{\psi_i}$ as quaternion algebras:
\begin{align*}
&\calA_{\psi_1} =\left(\frac{-1,x-7}{\Q(E)}\right),\quad
\calA_{\psi_2} =\left(\frac{-11,11x-76}{\Q(E)}\right),\\
&\calA_{\psi_3} =\left(\frac{-15,15x-104}{\Q(E)}\right),\quad
\calA_{\psi_4} =\left(\frac{-19,19x-132}{\Q(E)}\right).
\end{align*}

\begin{proof}[Proof of \Cref{thm:cremona_mazur_curves_can_have_large_genus}]
In the quaternion algebra $\calA_{\psi_1}$ we have the following identity:
\[g(x)\coloneqq x^3 - 26x^2 + 233x - 704=((x-5)i+(x-9)j+4k)^2.\]
This was found by a search as described below. The square of the quaternion $ai + bj + ck$ is $-a^2+(x-7)b^2-(x-7)c^2$. For this to have degree $3$ polynomial, the polynomials $a, b, c \in K[x]$ must have degree $\leq 1$.
We then searched for appropriate coefficients so that the elliptic curve $y^2 = g(x)$ has rank $0$ and trivial two-torsion.

Consider the quadratic extension of $\Q(E)$ given by $L=\Q(E)(\sqrt{g(x)})$. The de Jong construction associates a double cover $X \rightarrow E$, ramified at the zeros of $g$, whose Jacobian visualizes $\sigma_1$. By Riemann-Hurwitz, we know that the genus of $X$ is 4. Consider the elliptic curve
$E_1\colon y^2=g(x)$,
and the hyperelliptic curve 
$X_1\colon y^2=g(x)f_E(x)$,
where $f_E(x) = x^3-1204416x-508759920$ is a polynomial defining a short Weierstrass model of $E$.
Since $L$ is a biquadratic extension of the rational function field $\Q(x)$, we have the diagram of subfields and corresponding maps between curves
\begin{center}
    \begin{tikzpicture}
        \node (K) at (0,2) [minimum size=6mm] {$\Q(E)(\sqrt{g(x)})$};
        \node (A) at (-2,1) [minimum size=6mm] {$\Q(E)$};
        \node (B) at (0,1) [minimum size=6mm] {$\Q(E_1)$};
        \node (C) at (2,1) [minimum size=6mm] {$\Q(X_1)$};
        \node (Q) at (0,0) [minimum size=6mm] {$\Q(x)$};
        
        \node (K1) at (6,2) [minimum size=6mm] {$X$};
        \node (A1) at (4,1) [minimum size=6mm] {$E$};
        \node (B1) at (6,1) [minimum size=6mm] {$E_1$};
        \node (C1) at (8,1) [minimum size=6mm] {$X_1$};
        \node (Q1) at (6,0) [minimum size=6mm] {$\P^1$};
        
        \draw (Q) -- (A);
        \draw (Q) -- (B);
        \draw (Q) -- (C);
        \draw (A) -- (K);
        \draw (B) -- (K);
        \draw (C) -- (K);

        \draw[->] (A1) -- (Q1);
        \draw[->] (B1) -- (Q1);
        \draw[->] (C1) -- (Q1);
        \draw[->] (K1) -- (A1);
        \draw[->] (K1) -- (B1);
        \draw[->] (K1) -- (C1);
    \end{tikzpicture}
\end{center}
which give the isogeny decomposition $\Jac(X) \sim E\times E_1\times \Jac(X_1)$. Note that the isogeny is constructed by gluing the two abelian surfaces $\Jac(X_1)$ and $E \times E_1$ along $2$-torsion, using the $G_\Q$-isomorphism $\Jac(X_1)[2]\cong E[2]\times E_1[2]$.

Let $T$ be the submodule of $\Jac(X_1)[2]$ isomorphic to $E[2]$. Let $\Delta \subset E[2]\times \Jac(X_1)[2]$ be the graph of this isomorphism, and $A$ denote the subabelian variety $E\times \Jac(X_1)/\Delta$ of $\Jac(X)$. We claim that
$$\Vis_{A} H^1
(K, E)= \Vis_{\Jac(X)} H^1
(K, E),$$ 
which implies that $\sigma_1\in\Vis_{A} H^1
(K, E)$, and consequently $\Jac(X)$ is not minimal in $\V(E,\sigma_1)$. By \Cref{obs1}, it suffices to prove that
$$\Vis_{\Jac(X)} H^1
(K, A)=(0).$$

Consider the diagram 
\[
    \xymatrix{&&0\ar[d]&&\\
    &&E_1\ar[d]\ar^{\psi}[dr]&&\\
    0\ar[r]& A\ar_{\phi}[dr]\ar[r]&\Jac(X)\ar[d]\ar[r]&E_1'\ar[r]&0\\
    &&A'\ar[d]&&\\
    &&0,&&}
    \]
where $E_1'= \Jac(X)/A$ and $A' = \Jac(X)/E_1$. The row and column are exact sequences of abelian varieties, and the diagonal maps $\phi$
and $\psi$ are isogenies. We know that the kernel of the isogeny $\psi$ is $E_1[2]$, and thus $\psi = \pm 2$. As in the proof of {\cite[Lemma 2.1]{Fisher_vis7}} we have the following exact sequence:
\begin{equation}\label{Fisher ES}
0 \longrightarrow
\frac{\Jac(X)(\Q)}{E_1(\Q)+A(\Q)}
\longrightarrow
\frac{E_1'(\Q)}{\psi E_1(\Q)}
\longrightarrow \Vis_{\Jac(C)} H^1
(K, A) \longrightarrow 0.    
\end{equation}
Since $E_1$ is the elliptic curve \href{https://www.lmfdb.org/EllipticCurve/Q/26/a/3}{26.a3}, we know $E_1(\Q)\cong\Z/3\Z$, and consequently we obtain that
\[
\frac{E_1'(\Q)}{\psi E_1(\Q)} = \frac{E_1(\Q)}{2 E_1(\Q)} = (0).
\]
By the exact sequence \eqref{Fisher ES}, we obtain the desired result, which proves \Cref{thm:cremona_mazur_curves_can_have_large_genus}.
\end{proof}

\begin{proof}[Proof of \Cref{thm:categories_not_equal}]
By \Cref{thm:explicit_cremona_mazur_for_n_2} applied to the binary quartic model $\psi_2$ of $\sigma_2$, we know that $C_2$ is the hyperelliptic curve given by the equation
\[
y^2=1/1331 x^6 - 2324/1331 x^4 + 1136/1331 x^2 - 144/1331.
\]
The Jacobian of $X_2$ is isogenous to the product $E\times F$, where $F$ is the elliptic curve \href{https://beta.lmfdb.org/EllipticCurve/Q/27408/d/1}{27408.d1}. We have $F(\Q)\cong \Z$.
By \Cref{lemma:fisher_2.1}, $\#\Vis_{\Jac(X_2)} H^1
(\Q, E)=\#(F'(\Q)/\psi F(\Q))=\#(F(\Q)/2F(\Q)=2$. Furthermore, by construction, the element $\sigma_2\in \Sha(E/\Q)$ associated with $\psi_2$ is in $\Vis_{\Jac(X_2)}(\Sha(E/\Q))$. Consequently,
\[
2=\#\langle \sigma_2\rangle\leq \#\Vis_{\Jac(X_2)}(\Sha(E/\Q))\leq \#\Vis_{\Jac(X_2)} H^1
(\Q, E)=2,
\]
which implies that $\sigma_2$ is the only nontrivial element visualized by $\Jac(X_2)$. We have constructed an abelian variety that visualizes a nontrivial element of $\Sha(E/\Q)$ but not the whole group. In other words, $\Jac(X_2)\in \V(E,\sigma_2)$, but $\Jac(X_2)\notin \V(E,2)$.
\end{proof}

\begin{example}
    Continuing with the notations from the beginning of \Cref{sec:cremona_mazur_curves_can_have_large_genus}, let us look more carefully at the explicit visualization construction of \Cref{thm:explicit_cremona_mazur_for_n_2} for the binary quartic model $\psi_1$ of $\sigma_1 \in \Sha(E)$.
    Then $X_1$ is the hyperelliptic curve 
    \[
    y^2=1/64 x^6 - 211/16 x^4 - 29/4 x^2 - 1,
    \]
    which is isomorphic to the curve with LMFDB label \href{https://beta.lmfdb.org/Genus2Curve/Q/326041/a/326041/1}{326041.a.326041.1}. The Jacobian of $X_1$ is isogenous to the product $E\times F$, where $F$ is the elliptic curve \href{https://www.lmfdb.org/EllipticCurve/Q/571/a/1}{571.a1}. We have $F(\Q)\cong \Z^2$, and the minimal discriminants of $E$ and $F$ are equal to $-571$.
    While \Cref{thm:fisher_2.2} cannot be applied because $p=2$, Fisher already shows in \cite[Section 15.1]{Fisher_hess12} that $F(\Q)/2F(\Q) \simeq \Sha(E)[2]$.
    We can also prove this by a straightforward computation in $\mathcal{A}_{\psi_2}$:
    \[
    \left(\frac{1}{11}(i+k)\right)^2=\frac{1}{121}(i^2+k^2)=\frac{1}{121}(-11+11(11x-76))=x-7.
    \]
    By \cite[Lemma 5.4.7]{Voight}, the algebra $\mathcal{A}_{\psi_2}\otimes \Q(E)(\sqrt{x-7})$ splits. This implies that the element $\sigma_2\in \Sha(E/\Q)$ associated to $\psi_2$ lies in $\Vis_{\Jac(X_1)}(\Sha(E/\Q))$. By construction $\sigma_1 \in \Vis_{\Jac(X_1)}(\Sha(E/\Q))$. Consequently, $\Jac(X_1)\in \V(E,2)$.
\end{example}

This example raises a concrete question about uniform visibility of $\Sha(E)[2]$.
\begin{question}
\label{question:uniform_visibility}
    Let $E/\Q$ be an elliptic curve with $\Sha(E)[2] \ne \{0\}$. Suppose $\Sha(E)[2] \simeq (\Z/2)^n$. Let $\psi_i$ for $1 \leq i \leq 2^n-1$ denote minimised and reduced binary quartic models representing the non-zero elements of $\Sha(E)[2]$. Then, does there exist an $i$ such that the genus $2$ Jacobian constructed from $\psi_i$ in \Cref{thm:explicit_cremona_mazur_for_n_2}, visualizes all of $\Sha(E)[2]$?
\end{question}

\printbibliography

\end{document}